\documentclass[preprint]{imsart}
\RequirePackage{amsthm,amsmath,amsfonts,amssymb}
\RequirePackage[numbers]{natbib}

\startlocaldefs

\usepackage{amsmath}
\usepackage{amssymb, latexsym, amsfonts, amscd, amsthm, mathrsfs, enumerate, esint}
\usepackage[usenames,dvipsnames]{color}
\usepackage[all]{xy}
\usepackage{graphicx}
\usepackage{epsfig}
\usepackage{hyperref}

\numberwithin{equation}{section}

\definecolor{purple}{rgb}{0.9,0,0.8}

\definecolor{gray}{rgb}{0.7,0.7,0.7}

\newtheorem{thm}{Theorem}[section]

\newtheorem{lem}[thm]{Lemma}

\theoremstyle{definition}

\newtheorem{remark}[thm]{Remark}

\newcommand{\beq}{\begin{equation}}
\newcommand{\eeq}{\end{equation}}

\newcommand{\E}{\mathbb{E}}

	\renewcommand{\P}{\mathbb{P}}

\newcommand{\R}{\mathbb{R}}

\usepackage{verbatim}

\usepackage{color}

\def\R{\mathbb{R}}

\endlocaldefs

\begin{document}

\begin{frontmatter}
\title{Large Deviation Principle for the Whittaker 2d Growth Model}
\runtitle{Large Deviation Principle}

\begin{aug}
\author[A]{\fnms{Jun} \snm{Gao}\ead[label=e1]{jung2@stanford.edu}},
\author[B]{\fnms{Jie} \snm{Ding}\ead[label=e2]{dingj@umn.edu}}
\address[A]{Department of Mathematics, Stanford University, 
\printead{e1}}

\address[B]{School of Statistics, University of Minnesota 
\printead{e2}}
\end{aug}


\begin{abstract}
The Whittaker 2d growth model is a triangular continuous Markov diffusion process that appears in many scientific contexts. It has been theoretically intriguing to establish a large deviation principle for this 2d process with a scaling factor. The main challenge is the spatiotemporal interactions and dynamics that may depend on potential sample-path intersections. We develop such a principle with a novel rate function. Our approach is mainly based on Schider's Theorem, contraction principle, and special treatment for intersecting sample paths.
\end{abstract}


\end{frontmatter}

\section{Introduction}


The Whittaker 2d growth model, denoted by
$$ T( t ) =\{T_{k,j}(t)\}_{1 \leq j \leq k \leq N} $$
and valued in $\R^{N(N+1)/2}$,
is temporally a continuous Markov diffusion process  and spatially a triagular array of particles.
In particular, it is defined by the following system of stochastic differential equations (SDEs).
\begin{align*}
&\hspace{-2cm}\textrm{ For } k=1: \\
dT_{1,1}&=dW_{1,1}+ a_1dt, \\
&\hspace{-2cm}\textrm{ For } k=2, \ldots , N: \\
dT_{k,1}&= dW_{k,1}+(a_k+ e^{T_{k-1, 1}- T_{k,1}})dt, \\
dT_{k,2}&= dW_{k,2}+(a_k+e^{T_{k-1,2}-T_{k,2}}-e^{T_{k,2}-T_{k-1,1}}) dt,  \\
&\, \cdots  \\
dT_{k,k-1}&= dW_{k,k-1}+(a_k+e^{T_{k-1,k-1}-T_{k,k-1}}-e^{T_{k,k-1}-T_{k-1,k-2}})dt, \\
dT_{k,k}&=dW_{k,k}+(a_k-e^{T_{k,k}-T_{k-1,k-1}})dt,
\end{align*}
where $\{W_{n,k}; 1 \leq k \leq n \leq N\}$ are independent Brownian motions.
In this work, we will only consider the non-drift case: $a_k = 0 $ for all $k$. %

The Whittaker 2d growth model can be described by a diffusion term of Brownian motion and a drift term of interaction with at most two particles on the layer above. The drift terms of interactions are one-directional: meaning the $k-1$ level can influence the $k$ level but not vice versa. Intuitively, the particles above are much ``heavier'' than those below.

The Whittaker 2d growth model is an important model to study as it is closely connected with several other fundamental stochastic processes.
The Whittaker 2d growth model was originally introduced in \cite{O} as a symmetric version of a stochastic differential system closely connected to the quantum toda lattice.  
Moreover, both the joint distribution and the infinitesimal generator of the models can be formulated via Whittaker functions, which are the eigenvalues of the Hamiltonian of a quantum toda lattice  
\cite{PS}.
Also, the Whittaker 2d growth model is the scaling limit of the q-Whittaker 2d growth model, which is a continuous-time and discrete-space stochastic process 
constructed from Macdonald symmetric functions~\cite{BC}.

 
\textbf{Contribution}.
The Whittaker 2d growth model is a hierarchical stochastic differential equation system that involves highly nontrivial interactions of particles. The standard large deviation principle (LDP) for independent Brownian motions cannot apply. This is mainly due to the technical challenge caused by possible intersection of sample paths, where the Schider's Theorem cannot apply. 
To address this, we need to apply sophisticated analysis and constructions of the rate function by approximated processes.
Also, we adopt a recursive proof to address the challenging effect of multi-layer interactions.
We derive the LDP local estimates in a hierarchical manner and complete the full LDP by showing the exponential tightness.

\textbf{Outline}.
The organization of the paper is given below.
In Section~\ref{sec_main}, we introduce the main result regarding the LDP of scaled Whittaker 2d growth models. 
In Section~\ref{sec_proof}, we provide proof of the main results. 
In particular, Subsection~\ref{sec_local_ldp} addresses the issue related to the boundary terms of the stochastic differential equation system. 
In Subsection~\ref{almost_interlaced}, we prove that particles are almost interlaced, and thus we can approximate the particles with processes found in local LDP. Based on the above development, we analyze a special four-particles system and then extend the result to general cases.
Subsection~\ref{main_theorem} proves the large deviation principle for the most general case and the good rate function.
In Appendix~\ref{sec_example}, we elaborate on two special cases of the theorem: the crossing case and strict interlacing case.
In Appendix~\ref{sec_existence}, we show the existence and uniqueness of the strong solution to the Whittaker 2d growth model.

\section{Main Result: LDP for Whittaker 2d growth model} \label{sec_main}
\label{section2}
For the Whittaker 2d growth model, we replace
$ T_{k,j}(t) $
by
$ \frac{1}{\gamma}T_{k,j}(\gamma t) $,
where $\gamma$ is a positive scaling parameter.
The SDE system becomes:

$$ dT_{1,1}=\frac{1}{\sqrt {\gamma}}dW_{1,1}$$

For $k=2, 3. \dots , N$:

\begin{align*}
dT_{k,1}&=\frac{1}{\sqrt {\gamma}} dW_{k,1}+ e^{\gamma({T_{k-1, 1}- T_{k,1}})}dt, \\
dT_{k,2}&= \frac{1}{\sqrt {\gamma}}dW_{k,2}+(e^{ \gamma (T_{k-1,2}-T_{k,2} )}- e^{ \gamma (T_{k.2}-T_{k-1,1})})dt, \\
\dots, \notag \\
dT_{k,k-1}&= \frac{1}{\sqrt {\gamma}}dW_{k,k-1}+(e^{\gamma (T_{k-1,k-1}-T_{k,k-1})}-e^{ \gamma (T_{k.k-1}-T_{k-1,k-2})})dt, \\
dT_{k,k}&= \frac{1}{\sqrt {\gamma}}dW_{k,k}-e^{\gamma (  T_{k,k}-T_{k-1,k-1} ) } dt.
\end{align*}
where $ \{W_{k,j} \}_{1 \leq j \leq k \leq N}$ is a set of independent standard Brownian motion,
with the fixed starting positions
\begin{align}
T_{k+1,j+1}(0) \leq T_{k,j}(0) \leq T_{k+1,j}(0), 1 \leq j \leq k \leq N.\label{eq_init}
\end{align}
We make the above assumption on initial positions so that the process does not diverge. 
In line with the above condition, we will study the sample paths (functions of time $t \in [0,1]$) satisfying
\begin{equation}
\phi_{k+1,j+1} \leq \phi_{k,j} \leq \phi_{k+1,j}, 1 \leq j \leq k \leq N.
\label{eq_sample_paths}
\end{equation}
in the study of LDP.

The following theorem introduces the large deviation principles for the scaled Whittaker 2d growth model and the rate function.

\begin{thm}[Main Theorem]
\label{mainthm}
Assume that the particles' initial condition (\ref{eq_init}) and sample path condition (\ref{eq_sample_paths}) hold.
The Whittaker 2d growth model satisfies LDP in $C_0([0,1])$ with the good rate function
\begin{equation}
I(\phi) = \sum_{1 \leq k \leq n \leq N} I_{n,k}(\phi_{n,k}),
\end{equation}
\begin{eqnarray*}
I_{n,k}(\phi_{n,k}) = \left\{ \begin{array}{ll}
\sum_{1 \leq k \leq n \leq N} I_{n,k}(\phi_{n,k})
& \quad \textrm{ if  $\phi_{k+1,j+1} \leq \phi_{k,j} \leq \phi_{k+1,j}, 1 \leq j \leq k \leq N.$ }  \\
+\infty & \quad\textrm{ otherwise.}
\end{array}
\right.
\end{eqnarray*}
where $ \phi = \{ \phi_{n,k} \in C_0[0,1]; 1 \leq k \leq n \leq N \},$
$C_0([0,1])$ is the space of continuous functions on the interval $[0,1]$ with the $\mathcal{L}_\infty$ norm,
and


\begin{eqnarray*}
I_{n,k}\{\phi_{n,k}) =
\left\{ \begin{array}{ll}
\biggl\{\frac 1 2 \int_{\phi_{n-1,k-1} > \phi_{n,k}>\phi_{n-1,k}}\dot{\phi}_{n,k}^2 ds  \\
\quad+ \frac 1 2 \int_{\phi_{n-1,k-1}= \phi_{n,k}> \phi_{n-1,k}}(\dot{\phi}_{n,k})_-^2  \nonumber \\
\quad+\frac 1 2 \int_{\phi_{n-1,k-1} > \phi_{n,k} = \phi_{n-1,k}}(\dot{\phi}_{n,k})_+^2\biggr\}
& \quad \textrm{ if $\phi_{n,k} \in \mathcal{AC}$ and $\phi_{n,k}(0) = T_{n,k}(0),$}  \\
+\infty & \quad\textrm{ otherwise.}
\end{array}
\right.
\end{eqnarray*}
Here, we define $(x)_+ = max(x,0),(x)_- = max(-x,0) $, and $\mathcal{AC}$ denotes the set of absolutely continuous functions on $[0,1].$

\end{thm}


\begin{remark}
If there is no $t$ such that $\phi_{n-1,k-1}(t) = \phi_{n,k}(t)$ and $\phi_{n,k}(t) = \phi_{n-1,k}(t)$, which means $\phi_{n-1,k-1}< \phi_{n,k}$ and $\phi_{n,k} < \phi_{n-1,k}$. Then the rate function degenerates to following: 
\begin{eqnarray*}
I(\phi) = 
\frac{1}{2}
\left\{ \begin{array}{ll}
\int_0^1 \dot \phi ^ 2 ds, & \quad \textrm{{if $\phi \in \mathcal{AC}$ and $\phi(0) = T(0)$}} \\
+\infty & \quad \textrm{ otherwise,}
\end{array}
\right.
\end{eqnarray*}
which is the rate function of a single scaled Brownian motion $\frac{1}{\sqrt{\gamma}}W(t).$
So we can regard this result as an extention of a similar rate function for a single particle, except for the areas where sample paths intersect.
\end{remark}

\begin{remark}

When $\phi_{n,k}$ and $\phi_{n,k-1}$ intersect, meaning that there is $t$ such that $\phi_{n,k}(t) =\phi_{n,k-1}(t),$ the probability of two particles staying within these two sample paths' neighbourhood should be higher than that in the Brownian motion case. That is why in the part $\phi_{n,k}(t) =\phi_{n,k-1}(t)$, the rate function is different from that of a Brownian motion. 

We also made assumptions on initial positions, in the same order that we require the sample paths. Otherwise the sample paths will be crossing and then the rate function is $\infty$, as proved in Appendix \ref{existence}.
 We consider the space of absolutely continuous functions $\mathcal{AC}$, since the sample paths in $\mathcal{AC}$ are differentiable almost everywhere and the rate function is well-defined.

\end{remark}



\section{Proof of Main Theorem} \label{sec_proof}


We first introduce a sketch of the proof.
In Subsection 3.1 (``local LDP''), we first prove for a single particle on a given sub-interval of $[0,1].$
To simplify the notation, we consider the following generic form of a single $\{T_{n,k}\}$,
$$dT = dW + ( e ^ {T^{(-)} -  T }  - e^{ T - T^{(+)}} ) dt,$$
and let $\phi, \phi^{(+)},\phi^{(-)} $ denote the sample paths of $T, T^{(+)},T^{(-)} $, respectively. 
We start with the simplest case of one sample path, say $\phi$, and show that $\phi^{(+)}$ is strictly detached from $\phi,$ and $\phi^{(-)}$ coincides with $\phi$ for some part of the time interval. 

In Subsection 3.2 (``almost interlaced''), we prove the particles are interlaced with a high probability given the initial conditions. This part of proof enables us to apply the results in Subsection 3.1 to Subsection 3.3. 

In Subsection 3.3 (``general case''), we complete the main proof. 
We start from the special case of a four-particle system. We segment the interval into small intervals and apply our results from Subsection 3.1. Then we construct the lower and upper bounds of the rate function. By an appropriate choice of independent parameters in the bounds, we obtain the exact rate functions.
Finally, we extend our analysis to the general system with similar proofs. 

\subsection{Local LDP}
\label{sec_local_ldp}
Suppose that $a,b$ are constants that satisfy $0 \leq a < b \leq 1$. 
We are interested in the case:
\begin{equation}
\label{local_ldp_eq}
dT = \frac 1 {\sqrt \gamma} dW + (e^{\gamma(\phi^{(-)} - T)} - e^{\gamma (T - \phi^{(+)})})dt,  
\end{equation}
on $[a,b],$ with fixed starting positions that satisfy:
$$\phi^{(-)}(a) \leq T(a) < \phi^{(+)}(a). $$

We are interested in the convergence rate at $T=\phi$,
where $\phi^{(-)} \leq \phi < \phi^{(+)} \; on [a,b]. $

\begin{lem}
\label{lem3.1}
Recall the definition of $T$,
$$dT = \frac 1 {\sqrt \gamma} dW + (e^{\gamma(\phi^{(-)} - T)} - e^{\gamma (T - \phi^{(+)})})dt.$$
We define
\begin{equation}
\label{tilde_T_0}
d\widetilde T_0 = \frac 1 {\sqrt \gamma} dW + e^{\gamma(\phi^{(-)}- \widetilde T_0)} dt, \nonumber 
\end{equation}
with $\widetilde T_0(a) =  T(a).$
Then, $T$ and $\widetilde T_0$ are exponentially equivalent 
conditional on $\|T -\phi\| < \delta$ for a $\delta$ small enough.
\end{lem}

\begin{remark}
To explain exponential equivalence by its definition, we just need to prove for any $\delta  >0$ that
\begin{equation}
\limsup_{\gamma \to \infty}\frac 1 \gamma \log \P \bigl\{ \|T - \widetilde T_0 \| > \delta  \bigr\} = -\infty.\nonumber 
\end{equation}
\end{remark}

\begin{remark}
For the dual case
with fixed starting positions,
$$\phi^{(-)}(a) < T(a) \leq \phi^{(+)}(a),$$
we are interested in the convergence rate at $T=\phi$,
where $\phi^{(-)} < \phi \leq \phi^{(+)} \; on [a,b]. $
We will state the duel result in the end briefly, since the proof is parallel to the first case.
\end{remark}

\begin{lem}
\label{lem3.2}
Recall $\widetilde T_0$ defined in (\ref{tilde_T_0}). We define
\begin{equation}
\label{tilde_T}
d \widetilde T (t)= \frac 1 {\sqrt \gamma} dW + \frac{d}{dt} \sup_{[a,t]}(\phi^{(-)} -  \frac 1 {\sqrt \gamma} W)_+ dt,\nonumber 
\end{equation}
with $\widetilde T(a) = \widetilde T_0(a)$.
Then $\widetilde T_0 $ and  $\widetilde T$  are exponentially equivalent, conditioning on $\|\tilde T_0 -\phi\| < \delta$ where $\delta$ is small enough. 
\end{lem}

\begin{lem}
\label{lem3.3}
Assume that $\phi^{(-)} \leq \phi,on [a,b]. $
 The stochastic process $\widetilde T$ satisfies:
\begin{equation}
\label{3.8}
\widetilde T(t) = (\widetilde T-\frac 1 {\sqrt \gamma} W)(a) + \frac 1 {\sqrt \gamma} W(t) + \sup_{[a,t]} (\phi^{(-)} -  \frac 1 {\sqrt \gamma} W - (T -\frac 1 {\sqrt \gamma} W)(a)   )_+, \nonumber 
\end{equation}
on $[a,b]$, where $X = \frac 1 {\sqrt \gamma} W$ is a scaled Brownian motion.
So the rate function for $\widetilde T$  at $\{\widetilde T = \phi \geq \phi^{(-)}\}$ is
\begin{displaymath}
 J(\phi \mid \phi^{(-)})=
\left\{ \begin{array}{ll}
\frac 1 2 \int_{[a,b]\cap \{\phi=\phi^{(-)}\}} (\dot{\phi})_-^2
+\frac 1 2  \int_{[a,b] \cap \{\phi>\phi^{(-)}\}} \dot{\phi}^2
& \quad\textrm{if $\phi \in \mathcal{AC},$  }  \\
\infty & \quad\textrm{otherwise.} \end{array} \right.
\end{displaymath}
Furthermore, since $T$ and $\widetilde T$ are exponentially equivalent, $T$ has the same rate function.
\end{lem}

To obtain some intuitions about the above lemmas, we note the definition of $\tilde{T}_0$ in (\ref{tilde_T_0}) and transform it into following:
\begin{equation}
d( \tilde T_0 - \frac{1}{\sqrt{\gamma}} W) = e^{\gamma (\phi^{(-)} - \frac{1}{\sqrt{\gamma}}W)  - \gamma ( \tilde T_0 - \frac{1}{\sqrt{\gamma}}W) } dt \label{eq3_6}.
\end{equation}
To simplify the notation, we denote $\tilde T_0 - \frac{1}{\sqrt{\gamma}} W$ as $X_\gamma$ and $\phi^{(-)}- \frac{1}{\sqrt{\gamma}} W$ as $Y_\gamma$.
Then we rewrite (\ref{eq3_6}) as
\begin{equation}
dX_\gamma = e^{\gamma(Y_\gamma-X_\gamma)}dt,\nonumber 
\end{equation}
where $X_\gamma(a) = 0$ and $Y_\gamma (a) \leq 0$.
We then observe that
\begin{equation}
\lim_{\gamma \to +\infty} \big( X_\gamma(t) - \sup_{[0,t]}( Y_\gamma )_+ \big) = 0.\nonumber 
\end{equation}
in $L^\infty$ norm.

\begin{lem}
\label{similar_proof}
Consider the dual case $\phi^{(+)} \geq \phi$ on $[a,b]. $
 The stochastic process $T$ satisfies:
\begin{equation}
\widetilde T(t) = (\widetilde T-X)(a) + X(t) - \sup_{[a,t]} \biggl\{-(\phi^{(+)} - X - (\widetilde T -X)(a)   )\biggr\}_+,\nonumber 
\end{equation}
on $[a,b]$, where $X = \frac 1 {\sqrt \gamma} W$ is a scaled Brownian motion.
So $T$ has the convergence rate at $\{\widetilde T = \phi \leq \phi^{(+)}\}$:
\begin{displaymath}
 J(\phi| \phi^{(+)})=
\left\{ \begin{array}{ll}
\frac 1 2 \int_{[a,b]\cap \{\phi=\phi^{(+)}\}} (\dot{\phi})_+^2
+\frac 1 2  \int_{[a,b] \cap \{\phi<\phi^{(+)}\}} \dot{\phi}^2
&\quad \textrm{if $\phi \in \mathcal{AC},$  }  \\
\infty &\quad \textrm{otherwise.} \end{array} \right.
\end{displaymath}
Furthermore, since $T$ and $\widetilde T$ are exponentially equivalent, $T$ has the same rate function.
\end{lem}

\begin{proof}[Proof of Lemma \ref{lem3.1}]
Since $T < \phi^{(+)} \; on [a,b]$, there is $\eta > 0, $ such that
for $\|T-\phi\|\leq \delta,\|T_\pm-\phi_\pm \|\leq \delta,$
$$ T-\phi^{(+)} \leq - \eta.  $$
Let $\mu(t) \doteq \int_a^t e^{\gamma(T - T_+)} ds,$
$$d(T+\mu ) = \frac 1 {\sqrt \gamma} dW + e^{\gamma(T_-+\mu - (T+\mu))}dt.$$

If we prove
\begin{equation}
\label{3.2}
|\widetilde T_0(t) -T(t) | \leq  \mu (t),\nonumber 
\end{equation}
then  $$\|\widetilde T_0 -T\| \leq \sup_{[a,b]} \mu \leq e^{-\frac {\gamma \eta} 2 }(b-a),$$
and therefore the exponential equivalence is proved.

To prove (\ref{3.2}), we compare the formula for $\widetilde T_0 $ and $T+\mu.$ Comparing the right side of $T$ and $T_0$, by comparison principle, we have $T(t) \leq \tilde T_0(t).$ Comparing the right side of $T+\mu$ and $T_0$, by comparison principle, we have
$T(t)+ \mu(t) \geq \widetilde T_0(t). $
Therefore, we give an upper bound for the differences, which is exactly (\ref{3.2}).
\end{proof}

\begin{proof}[Proof of Lemma \ref{lem3.2}]
For $$d\widetilde T_0 = \frac 1 {\sqrt \gamma}  dW + e^{\gamma (\phi^{(-)}-\widetilde T_0)}dt,$$ we directly solve it as
$$\widetilde T(t) =\biggl(T-\frac 1 {\sqrt \gamma}W\biggr)(a) + \frac 1 {\sqrt \gamma}W(t) + \frac{1}{\gamma} \log\biggl\{ 1+ \gamma \int_a^t e^{\gamma (\phi^{(-)} - \frac 1 {\sqrt \gamma}W )(s) - \gamma (T-\frac 1 {\sqrt \gamma}W)(a)} ds  \biggr\}.$$
We denote the drift term as
\begin{align*}
V_\infty (t) &= \sup_{s \in [a,t]} \biggl(\phi^{(-)}(s) - \frac 1 {\sqrt \gamma}W(s)\biggr)_+, \\
V_\gamma (t) &= \frac{1}{\gamma} \log\biggl\{ 1+ \int_a^t \gamma e^{\gamma(\phi^{(-)}(s) - \frac 1 {\sqrt \gamma}W(s))}ds \biggr\}.
\end{align*}
Note that $\widetilde T - \widetilde T_0 = V_\infty - V_\gamma.$ Next, we compare the two drift terms, and choose $\gamma $ large enough so that $\frac{1}{\gamma} \log(1+ \gamma) \leq \frac{\delta}{2}$, and
\begin{align}
\label{V_upper_bound}
V_\gamma (t) &\leq V_\infty(t) + \sup_{s \in [a,t]} \frac{1}{\gamma}  \log\biggl\{ e^{-\gamma \sup_{s\in [a,t]}(\phi^{(-)}(s) - \frac 1 {\sqrt \gamma}W(s))_+ } \gamma (b-a) \biggr\} \nonumber \\
&\leq V_\infty(t) +  \frac{1}{\gamma} \log\{1+ \gamma \} \leq \frac \delta 2.
\end{align}


To obtain an inequality of the opposite direction, we let 
$$A_\epsilon \doteq \biggl\{t: (\phi_- - \frac 1 {\sqrt \gamma} W)_+ \geq \sup_{[a,t]}(\phi^{(-)} - \frac 1 {\sqrt \gamma}  W )_+ - \epsilon    \biggr\} \subset [0,1].$$ We have the following estimates
\begin{eqnarray}
V_\gamma(t) \geq V_\infty(t) + \frac{1}{\gamma} \log |A_\epsilon| - \epsilon. \notag
\end{eqnarray}
 Since $\phi^{(-)} \in C_0([0,1]),$ we can choose $\delta >0,$ $\gamma $ large enough so that $$\sup_{|s-t| \leq e^{-\frac {\gamma\delta} 4}} |  \phi^{(-)}(t)  - \phi^{(-)}(s)| \leq \frac{\epsilon}{2}. $$
Therefore,
\begin{align}
\label{V_lower_bound}
\P\biggl\{\frac{1}{\gamma} \log |A_\epsilon | \geq - \frac \delta 4 \biggr\} \notag
&\leq \P \biggl\{ \sup_{|s-t| \leq e^{-\frac {\gamma\delta} 4}} |  (\phi^{(-)} - \frac 1 {\sqrt\gamma} W)(t)  - (\phi^{(-)} - \frac 1 {\sqrt\gamma} W)(s)| \geq \epsilon   \biggr\}  \notag \\
&\leq \P \biggl\{ |\frac 1 {\sqrt\gamma} W(t) - \frac 1 {\sqrt\gamma} W(s)| \geq \frac{\epsilon}{2}   \biggr\} \nonumber \notag \\
&\leq a(\gamma) \doteq C \exp \biggl\{ \frac {\gamma\delta} 4 - \frac{\gamma \epsilon^2}{4}e^{\frac {\gamma\delta} 4}   \biggr\}.
\end{align}

Therefore, setting $\epsilon = \frac{\delta}{2}$ and applying (\ref{V_upper_bound}) and (\ref{V_lower_bound}), we obtain
\begin{align}
\label{3.4_v2}
V_\gamma (t) \leq V_\infty(t) +  \frac{1}{\gamma} \log\{1+ \gamma \} \leq \frac \delta 2, \\
\label{3.5_v2}
V_\gamma (t) \geq V_\infty(t) + \frac{1}{\gamma} \log |A_\epsilon| - \epsilon \geq \frac 3 4 \delta.
\end{align}
where (\ref{3.5_v2}) holds with probability $1- a(\gamma).$ Furthermore, combining (\ref{3.4_v2}) and (\ref{3.5_v2}), we have
$\|V_\gamma - V_\infty\|  \leq \delta,$ with probability $1-a(\gamma)$, and
\begin{equation*}
\limsup_{\gamma \to \infty} \frac 1 \gamma \log \P \{ \|\widetilde T - \widetilde T_0\| \leq \delta \} \leq \lim_{\gamma \to \infty} \frac 1 \gamma \log a(\gamma) = - \infty.
\end{equation*}
Hence, we proved the exponential equivalence of $T,\widetilde T_0, \widetilde T.$
\end{proof}

\begin{proof}[Proof of Lemma \ref{lem3.3}]
We have the equation:
\begin{equation}
\dot \phi = \dot \psi + \frac d {dt} \sup_{[a,t]} \biggl(\phi^{(-)} - \psi - (\phi - \psi)(a) \biggr)_+. \nonumber 
\end{equation}
Consider the equation as a mapping from $T = \phi,T_- = \phi_-$ to $X= \psi $, the rate function is

\begin{displaymath}
\left\{ \begin{array}{ll}
\frac{1}{2}\inf_{\phi(t) = (\phi-\psi)(a) + \psi(t) + \sup_{[a,t]} ( \phi_- - \psi - (\phi - \psi )(a)   )_+ }
{\int_a^b \dot{\psi}^2}
& \quad\textrm{if $\phi \in \mathcal{AC},$  }  \\
\infty & \quad\textrm{otherwise.} \end{array} \right. \nonumber 
\end{displaymath}

Case 1: On $$\biggl\{t:(\phi^{(-)} - \psi - (\phi-\psi)(a))_+(t) < \sup_{[a,t]}(\phi^{(-)} - \psi - (\phi-\psi)(a))_+\biggr\},$$
we have 
$$(\phi - \psi)(t) > (\phi-\psi)(a) + (\phi^{(-)} - \psi - (\phi-\psi)(a))_+ \geq (\phi^{(-)} - \psi)(t). $$
Therefore, $\phi(t) > \phi^{(-)}(t).$
Meanwhile, we have: $$\frac{d}{dt} \sup_{[a,t]}(\phi_- - \psi - (\phi-\psi)(a))_+ = 0, $$ and thus
$\dot \phi =\dot \psi.$

Case 2:
For the $t$ that belongs to the set
$$\biggl\{t: (\phi^{(-)} - \psi - (\phi-\psi)(a))_+ = \sup_{[a,t]}(\phi^{(-)} - \psi - (\phi-\psi)(a))_+ \biggr\}
\cap \biggl\{t: \phi^{(-)} - \psi - (\phi-\psi)(a) \geq 0\biggr\}, $$
we have
\begin{align*}
\phi(t) - \psi(t) 
&= (\phi-\psi)(a) + (\phi^{(-)} - \psi - (\phi-\psi)(a))_+ \nonumber \\
&=  (\phi-\psi)(a) + (\phi^{(-)} - \psi - (\phi-\psi)(a)) \\
&= \phi^{(-)}(t) - \psi(t),
\end{align*}
and therefore: $\phi = \phi^{(-)}.$

Since $\phi^{(-)} - \psi - (\phi-\psi)(a) > 0, $ and $\sup_{[a,t]}$ is non-decreasing, we have 
$$\frac{d}{dt} \sup_{[a,t]}(\phi^{(-)} - \psi - (\phi-\psi)(a))_+ \geq 0. $$
Therefore, $\dot \phi - \dot \psi \geq 0.$
When $\dot \phi < 0,$ $0 \geq \dot \phi = (\dot \phi)_- \geq \dot \psi,$  we have the inequality $\dot \psi ^2  \geq   (\dot \phi)_-^2.$
When $\dot \phi \geq 0,$ $(\dot \phi )_- = 0,$ we also have this inequality:  $\dot \psi ^2  \geq   (\dot \phi)_-^2.$
Combining these two, we have $\int \dot \psi ^2  \geq \int  (\dot \phi)_-^2. $

Case 3:
On the part $ \{t: (\phi^{(-)} - \psi - (\phi-\psi)(a))_+ = \sup_{[a,t]}(\phi^{(-)} - \psi - (\phi-\psi)(a))_+ \} \cap \{t: \phi^{(-)}(t) - \psi(t) - (\phi-\psi)(a) < 0 \},$
we have $$\phi^{(-)}(t)  - \psi(t) < (\phi - \psi)(a) = \phi(t) - \psi(t) , $$
and thus $\phi^{(-)}(t) < \phi(t), $ 
 $$ \frac{d}{dt} \sup_{[a,t]}(\phi^{(-)} - \psi - (\phi-\psi)(a))_+ = 0 $$
holds almost everywhere on this set, implying $\dot{\phi} = \dot{\psi}$.

Therefore, combining the three cases above, we obtain
\begin{eqnarray}
\int_a^b \dot{\psi}^2 \geq \int_{[a,b] \cap \{\phi^{(-)} = \phi \}}  (\dot \phi)_-^2
+ \int_{[a,b] \cap \{\phi^{(-)} <  \phi \}} \dot \phi^2.
\end{eqnarray}
Note that the equality can be achieved by setting
\begin{equation}
\psi(t) = \int_{a}^t ((\dot \phi )_- 1_{\phi=\phi^{(-)}}
+ \dot \phi 1_{\phi > \phi^{(-)}}) ds. \nonumber 
\end{equation}
Also, equation (\ref{3.8}) defines a continuous map from $X$ to $\widetilde T.$
By the contraction principle and Schider's Theorem, the result is proved.

\end{proof}

\begin{proof}[Proof of Lemma \ref{similar_proof}]
The proof is symmetric to proof of Lemma \ref{lem3.2}. We will skip here. 
\end{proof}

\subsection{The Almost Interlaced Particle Positions}
\label{almost_interlaced}
In this section, we show that
\begin{equation}
T_{n-1,k-1} \geq T_{n,k} \geq T_{n-1,k}.\nonumber 
\end{equation}
is "almost" true for all $(n,k).$ By "almost" we mean that the inequalities may be incorrect for most time, while, we can consider a small displacement: $f_n > 0$ so that
\begin{equation}
T_{n-1,k-1}+ f_n \geq T_{n,k} \geq T_{n-1,k} - f_n.\nonumber 
\end{equation}
Then the inequalities hold for most time for all $(n,k)$ with a high probabilities.
Also, we found that $f_n$ can be chosen as $f_n = \frac C {\sqrt \gamma}, $ which decreases to 0 as $\gamma \to \infty.$

We will first deal with the case for four particles, which is the case for the first few particles in a diamond: $T_{1,1}, T_{2,1}, T_{2,2}, T_{3,2}.$
For the general case of $\{T_{n,k}:1 \leq k \leq n \leq N\},$ the $N$ layers of particles have a very similar pattern to the first four particles, and thus the proof resembles the proof for the four particles.

\begin{remark}
For an event $A^\gamma$, we write $$\P\{ A^\gamma \} \sim e^{-\gamma \infty},$$
when we have $$- \lim_{\gamma \to \infty}\frac 1 \gamma \log \P\{A^\gamma\} = \infty.$$
\end{remark}

We define
\begin{align*}
A_n &= \{T_{n,k} - T_{n+1,k} \geq -f_n(\gamma), T_{n+1,k+1} - T_{n,k} \geq -f_n(\gamma), 1 \leq k \leq n\}, \\
B_n &= \{ T_{n+1,k}-T_{n,k}\geq -g_n(\gamma), T_{n,k} - T_{n+1, k+1}\geq - g_n(\gamma), 1 \leq k \leq n \}, \\
C_n &= \{T_{n+1,k} - T_{n+1,k+1} \geq -g_n(\gamma), 1 \leq k \leq n\}.
\end{align*}
We will show that 
\begin{align*}
&\P \{A_n ^c \} \sim e^{- \gamma \infty},\\
&\P \{C_n ^c \} \sim e^{- \gamma \infty},\\
&\P \{C_n \cap B_n^c \} \sim e^{- \gamma \infty}.
\end{align*}

\subsubsection{The Case for 4 Particles}

We address the scaled system of four particles on the time interval $[0,1]$, namely
\begin{align*}
dT_0 &= \frac{1}{\sqrt \gamma} dW_0, \\
dT_+ &= \frac{1}{\sqrt \gamma} dW_+ + e^{\gamma(T_0 - T_+)} dt, \\
dT_- &= \frac{1}{\sqrt \gamma} dW_- - e^{\gamma(T_- - T_0)}dt,\\
dT &=  \frac{1} {\sqrt \gamma} dW + (e^{\gamma(T_- - T)\} - g\{\gamma(T- T_+)})dt ,
\end{align*}
with the initial condition $T(0) = T_0(0) = T_+(0) = T_-(0) = 0. $


\begin{lem}
We let
$$g(\gamma) = 2 f(\gamma) = \frac{2}{\sqrt \gamma},$$ and
\begin{align*}
A &= \biggl\{T_+ - T_0 \geq - f(\gamma), T_0 - T_- \geq - f(\gamma)\; \biggr\}, \\
B &= \biggl\{ T_+-T \geq  - 2 g(\gamma),T-T_- \geq - 2 g(\gamma) \;   \biggr\}.
\end{align*}
Then we have
\begin{align}
\P (A^c) \sim e^{- \gamma \infty},    \notag \\
\P (A, B^c) \sim e^{- \gamma \infty}. \notag
\end{align}
\end{lem}

\begin{proof}
In the special case of $g(t) = e^t,$ we obtain the solution of $T_+$:
\begin{equation}
T_+
= \frac 1 {\sqrt \gamma}W_+ + \frac 1 \gamma \log\biggl\{ 1+ \int_0^t \gamma e^{\gamma(T_0 - \frac 1 {\sqrt \gamma} W_+) } ds \biggr\}. \notag
\end{equation}

\begin{equation}
\widetilde T_+
= \frac 1 {\sqrt \gamma}W_+ + \sup \biggl(T_0 - \frac{1}{\sqrt{\gamma}}W_+ \biggr)_+. \notag
\end{equation}

In the general case, notice $T_+$ and $\widetilde T_+$ are exponentially equivalent. We can replace $T_+$ by $\widetilde T_+.$

Therefore, we define
\begin{equation}
A_\epsilon = \biggl\{s \in [0,t]: T_0(s) -\frac{1}{\sqrt \gamma}W_+(s) - \sup_{[0,t]} (T_0 -\frac{1}{\sqrt \gamma}W_+)_+ \geq - \epsilon \biggr\},\notag
\end{equation}
and let $|A_\epsilon|$ denote the Lebesgue measure of $A_\epsilon$.
Set $\eta = e^{\gamma (- f + \epsilon - \frac{\log \gamma}{\gamma})}, \delta = \epsilon,$ we have 
\begin{align}
&\biggl\{T_+(t) - T_0(t) \leq - f \biggr\}
\subset \biggl\{ \log |A_\epsilon| \leq \gamma(-f + \epsilon) - \log \gamma  \biggr\} \notag \\
&\subset \biggl\{\sup_{t,s \in [0,1], |t-s| < \eta } |(T_0 - \frac 1 {\sqrt \gamma}W_+)(t) -(T_0 - \frac 1 {\sqrt \gamma}W_+)(s) | > \epsilon  \biggr\}. \nonumber 
\end{align}
Since $\frac {\sqrt \gamma} {\sqrt 2}(T_0 - \frac 1 {\sqrt \gamma}W_+)$ is a standard Brownian motion,
we have
\begin{align}
& \frac 1 \gamma \log \P\{T_+(t) - T_0(t) \leq - f \}
\leq \frac 1 \gamma \log\{\frac{1}{\epsilon \sqrt {\gamma \eta}}\} - \frac{\delta^{2}}{8 \eta}  \notag
\end{align}
Setting $\epsilon = f$,  
the right-hand side is $  \frac {\log {\gamma }} {3\gamma} - \frac{ \gamma ^ {\frac 1 3}} 8, $
which converges to $- \infty$ as $\gamma \to \infty$. Therefore, we have
\begin{equation}
\P \{\inf_{[0,1]}(T_+ - T_0) \leq -f(\gamma) = - \frac 1 {\sqrt \gamma}\} \sim e^{-\gamma \infty}.\nonumber 
\end{equation}
Similarly, we can prove
\begin{equation}
\P \{\inf_{[0,1]}(T_0 - T_-) \leq -f(\gamma) = - \frac 1 {\sqrt \gamma}\} \sim e^{- \gamma \infty}.\nonumber 
\end{equation}
Therefore, we conclude that
$\P \{A^c\} \sim e^{-\gamma \infty}.     $

For those sample paths which satisfy $\inf_{[0,1]}(T_+ - T_0) \leq -f(\gamma) = -2g(\gamma)$, there is a stopping time $s$ such that
$$(T-T_+)(s) = g(\gamma). $$
Recall that for $t>s,$
\begin{align}
&T_+ (t) -T_+(s)=  \frac 1   {\sqrt \gamma}( W_+ (t) - W_+(s))+ \int_s^t e^{ \gamma(T_0 - T_+) }. \nonumber 
\end{align}
Then we have the lower bound:
$$T_+ (t) -T_+(s) \geq  \frac 1   {\sqrt \gamma}( W_+ (t) - W_+(s)) .  $$
Note that in $A$,
$$T_+ - T_- \geq -2f(\gamma).$$
For $T - T_+ \geq g(\gamma)$, we have
\begin{align*}
&T_- - T \leq 2f(\gamma) - g(\gamma)=0 \nonumber \\
&\P \biggl\{A, \sup_{[0,1]}\{T-T_+\} \geq 2 g(\gamma) \biggr\} \nonumber \\
&\leq \P \biggl\{ \sup_{t \in [s,1]} \biggl\{   \frac 1 {\sqrt \gamma} (W(t)-W(s)) + \sup_{[s,1]} \{(-e^{\gamma g(\gamma) }+1)(t-s) - T_+\Big|^t_s \biggr\} \geq g(\gamma)     \biggr\}.
\end{align*}
We define
\begin{align*}
A_1(t) &= \frac 1 {\sqrt \gamma} (W(t)- W(s))+ \frac 1 2 (-e^{\gamma g(\gamma)}+ 1)(t-s), \\
A_2(t) &=  -\frac 1 {\sqrt \gamma} (W_+(t)- W_+(s))+ \frac 1 2 (-e^{\gamma g(\gamma)}+ 1)(t-s) .
\end{align*}
It follows from the above arguments that
\begin{equation}
\P\biggl\{ \sup_{[s,1]}\{T-T_+\} \geq 2 g(\gamma)\biggr\} \nonumber \\
\leq \P\biggl\{ \sup_{[s,1]}A_1 \geq \frac 1 2 g(\gamma) \biggr\} + \P\biggl\{ \sup_{[s,1]}A_2 \geq \frac 1 2 g(\gamma)\biggr\}.
\end{equation}
By the strong Markov property,
\begin{align}
&\P\biggl\{ A_1 \geq \frac 1 2 g(\gamma)  \biggr\}
\leq \P\biggl\{\sup_{[0,1]} A_1 |_{s=0} \geq \frac 1 2 g(\gamma) \biggr\}  \nonumber \\
\leq &\P\biggl\{  \sup_{[0, \gamma ^{-k}]} A_1 |_{s=0} \geq \frac 1 2 g(\gamma)   \biggr\} + \P\biggl\{  \sup_{[ \gamma ^{-k},1]} A_1|_{s=0} \geq  \frac 1 2 g(\gamma)  \biggr\}  .  \nonumber
\end{align}
For the first term,
\begin{align}
\P\biggl\{ \sup_{[0, \gamma ^{-k}]} A_1 \geq \frac 1 2 g(\gamma) \biggr\}
\leq &\P\biggl\{ \sup_{[0, \gamma ^{-k}]} \frac 1 {\sqrt \gamma}  W(t) \geq \frac 1 2 g(\gamma)  \biggr\}  \notag \\
\leq &\frac{4 \gamma^{\frac {1+k} 2}}{\sqrt {2 \pi }g(\gamma)} e^{-\frac 1 2 g(\gamma) \gamma^{1+k}}. \nonumber 
\end{align}

For the second term,
\begin{equation}
\P\biggl\{ \sup_{[\gamma^{-k},1]}A_1 \geq \frac 1 2 g(\gamma)   \biggr\}  \nonumber \\
\leq \frac 4 { (e^{\gamma g(\gamma) }-1) \gamma^{\frac 1 2 -k}}  \exp\biggl\{ - \frac 1 8 \gamma^{1-2k} (e^{\gamma g(\gamma) }-1)^2 \biggr\}.
\end{equation}
Setting $ k = \frac 1 2,$ by our choice of $f,g,$ we have:
$$\P\biggl\{\sup_{[s,1]} A_1  \geq \frac 1 2 g(\gamma)\biggr\}\sim e^{-\gamma \infty}.$$
Similar arguments apply to $A_2 \sim e^{-\gamma \infty} $, and  
$$\P\{A, \sup_{[0,1]}\{T-T_+\} \geq 2 g(\gamma)\}$$
has an infinite convergence rate. Also,
$$\P\biggl\{A, \inf_{[0,1]}\{T-T_-\}\leq -2g(\gamma) \biggr\}$$
has an infinite convergence rate.
Therefore, we conclude that
$$ \P\{A,B^c\} \sim e^{-\gamma \infty}. $$
\end{proof}

\subsubsection{General Case}
We showed that the particles are almost interlaced by the bounds $f(\gamma), g(\gamma).$
We will prove a similar result for all $\{T_{n,k}; 1\leq k \leq n \leq N\}. $
We write $T_0, T_+, T_-, T$  as $T_{1,1}, T_{2,1}, T_{2,2}, T_{3,2}$,
 and let $f_1(\gamma)= f(\gamma), g_1(\gamma)= g(\gamma).$
We let
 \begin{align*}
 A_1 &= \biggl\{T_{2,1} - T_{1,1} \geq -f_1(\gamma),T_{1,1} - T_{2,2} \geq -f_1(\gamma)   \biggr\},   \\
 B_1 &= \biggl\{ T_{2,1}-T_{3,2} \geq 2 g(\gamma),T_{3,2}-T_{2,2} \geq 2 g(\gamma) \quad  \biggr\}, \\
A_1 &\subset  C_1 = \{T_{2,1} - T_{2,2} \geq -g_1(\gamma)\}.
\end{align*}
Then,
\begin{align*}
\mathbb{P}\{ A_1^c  \}&\sim e^{-\gamma \infty}, \\
\mathbb{P}\{C_1^c\}&\sim e^{-\gamma \infty}, \\
\mathbb{P}\{C_1 \cap B_1^c\}&\sim e^{-\gamma \infty}.
\end{align*}
For
\begin{align*}
A_n &= \{T_{n,k} - T_{n+1,k} \geq -f_n(\gamma), T_{n+1,k+1} - T_{n,k} \geq -f_n(\gamma), 1 \leq k \leq n\}, \\
A_n &\subset C_n = \{T_{n+1,k} - T_{n+1,k+1} \geq -g_n(\gamma), 1 \leq k \leq n\}.
\end{align*}

We find $A_n \subset C_n$ if we set $2f_n(\gamma) = g_n(\gamma).$
Let $g_{n}(\gamma) = 4^{n-1}g(\gamma)$, and
\begin{equation}
B_n = \{ T_{n+1,k}-T_{n,k}\geq -g_n(\gamma), T_{n,k} - T_{n+1, k+1}\geq - g_n(\gamma), 1 \leq k \leq n \}. \nonumber 
\end{equation}
By a similar argument as the previous subsection, we can show
\begin{equation}
\mathbb{P}\{C_n \cap B_n^c\} \sim e^{-\gamma \infty}. \nonumber 
\end{equation}
Thus, the case is true for $A_{n+1}.$ By induction, we have all the relations. 

Moreover, we can replace $f_n, g_n$ by some positive bound $\delta,$ and introduce the events
$$A_n(\delta), B_n(\delta), C_n(\delta).$$
Since $f_n(\gamma), g_n(\gamma) \to 0,$ as $\gamma \to \infty,$
 $A_n, B_n, C_n$ are still dominant, meaning that
\begin{align*}
&\P \{A_n(\delta)^c\} \sim e^{-\gamma \infty}, \\
&\P \{C_n(\delta)^c\} \sim e^{-\gamma \infty},  \\
&\P \{C_n(\delta) \cap B_n(\delta)^c\} \sim e^{-\gamma \infty}.
\end{align*}

\subsection{Proof for the Main Theorem}
\label{main_theorem}
Similar to the arguments in previous subsections, we first show the result for the four particles.

\subsubsection{The Case for four Particles}
Recall the four particles we have in the previous section.
\begin{align*}
dT_0 &= \frac{1}{\sqrt \gamma} dW_0, \\
dT_+ &= \frac{1}{\sqrt \gamma} dW_+ + e^{\gamma(T_0 - T_+)} dt, \\
dT_- &= \frac{1}{\sqrt \gamma} dW_- - e^{\gamma(T_- - T_0)}dt,\\
dT &=  \frac{1} {\sqrt \gamma} dW + (e^{\gamma(T_- - T)} - e^{\gamma(T- T_+)})dt.
\end{align*}

It is even easier to take the first three particles into account, the rate function of which is easily computed.
Since by Lemma (\ref{lem3.1}) and (\ref{lem3.2}), we have shown that:
$T_{\pm}$ is exponentially equivalent to
\begin{equation}
\label{3.1}
\widetilde T_{\pm}(t) = T_{\pm}(0) + \frac 1 {\sqrt \gamma}W_{\pm}(t)  \pm \sup_{[0,t]} (\pm(T_0(s) - \frac 1 {\sqrt \gamma}W_{\pm}(s) - T_{\pm}(0)))_+.\nonumber 
\end{equation}
It is easy to show that the map from $ W_0,W_{\pm} $ to $T_0,\widetilde T_{\pm}$ is continuous from $(C_0([0,1]), \mathcal{L}_\infty)$ to itself.
Therefore, we can apply the contraction principle.

The difficulty lies in $T$. We will work on the limit of the probability
\begin{equation}
 \P\{\|T_{0, \pm} - \phi_{0,\pm}\|\leq \delta, \|T- \phi\|\leq \delta \},\nonumber 
\end{equation}
when $\gamma \to \infty,$ and $\delta \to 0.$
While its upper bound for the probability above is similar and even easier to get, we adopt the following strategies for the lower bound.

First, we note that the above equation is no smaller than
\begin{equation}
\lim_{\delta \to 0} \lim_{\delta_1 \to 0} \lim_{\gamma \to \infty} \P\{\|T_{0, \pm} - \phi_{0,\pm}\|\leq \delta_1, \|T- \phi\|\leq \delta \}.\nonumber 
\end{equation}
We consider $$\eta = \delta_2 - \delta_1 - \epsilon >0.$$

\subsubsection{Step 1: Fixing Intervals:}
We have previously proved that
\begin{align*}
A &= \biggl\{T_+ - T_0 \geq -\epsilon,T_0 - T_- \geq -\epsilon \;  \biggr\},   \\
B &= \biggl\{ T_+ - T \geq -\epsilon,T - T_- \geq -\epsilon \;  \biggr \}, \\
A &\subset  C = \biggl\{T_+ - T_- \geq -2\epsilon \;   \biggr\},
\end{align*}
$
\P (A^c) \sim e^{- \gamma \infty},
\P (C^c)  \sim e^{- \gamma \infty},
\P (C, B^c) \sim e^{-\gamma \infty}.
$
Hence these events $A,B,C$ are `dominant' meaning that the probabilities of $A,B,C$ will converge to one as $\epsilon \to 0.$

Consider the time interval
\begin{equation}
\label{3.3.3}
P\doteq \{t: \phi_+ - \eta \leq \phi \leq \phi_- + \eta\},
\end{equation}
which is a closed set. 
Its compliment in $(0,1),$ which is an open set, and thus a union of countable intervals: $P^c \doteq \cup_{n \in J} I_n.$
Denote $I_n = (a_n, b_n).$

Meanwhile, for $t \in P,$ we have
\begin{align*}
T(t) &> T_-(t) - \epsilon \geq \phi_-(t) -\delta_1 - \epsilon \geq \phi(t) -\delta_1 - \epsilon -\eta  =\phi(t) - \delta_2,\\
T(t) &< T_+(t) + \epsilon \leq \phi_+(t) +\delta_1 + \epsilon \leq \phi_+(t) +\delta_1 + \epsilon +\eta = \phi(t) + \delta_2.
\end{align*}
In other words, $|T- \phi|\leq \delta_2 < \delta$ holds on this part, and therefore we only need to consider $I_n$ for each $n$. We obtain the lower bound for the probability
\begin{align}
&\P\biggl\{\|T_{0, \pm} - \phi_{0,\pm}\|\leq \delta_1, \|T- \phi\|_{\cup I_n(\eta)}  \leq \delta \biggr\} \notag \\
&\geq \P\biggl\{A_\epsilon, \|T_{0, \pm} - \phi_{0,\pm}\|\leq \delta_1, \|T- \phi\|_{\cup I_n(\eta)}\leq \delta \biggr\} \nonumber   \\
\label{3.3.1}
&\geq \P\biggl\{A_\epsilon, \|T_{0, \pm} - \phi_{0,\pm}\|\leq \delta_1, \|T- \phi - (T-\phi)(a_n)\|_{ I_n(\eta)}\leq \delta- \delta_2, n \in J \biggr\} .
\end{align}
Then we replace $T - T(a_n)$ by $\widetilde T _n$. 
Next, we replace $\widetilde T _n$ by
\begin{equation}
\widetilde T _n(\phi) = dW_n + (e^{\gamma(\phi_- - T_n)} - e^{\gamma(T_n - \phi_+)})dt \nonumber .
\end{equation}
Note that
\begin{equation}
|\widetilde T _n - \widetilde T _n (\phi)| \leq 2\delta_1.\nonumber 
\end{equation}
We then have
\begin{align}
&\P\biggl\{\|T_{0, \pm} - \phi_{0,\pm}\|\leq \delta_1, \|T- \phi\|_{\cup I_n(\eta)}  \leq \delta \biggr\} \notag \\
&\geq \P\biggl\{A_\epsilon, \|T_{0, \pm} - \phi_{0,\pm}\|\leq \delta_1, \|\widetilde T _n (\phi_{\pm}) - (\phi-\phi(a_n))\|_{ I_n(\eta)}\leq \delta - 2 \delta_1 - \delta_2, n \in J \biggr\} \notag \\
&\geq \P\biggl\{A_\epsilon, \|T_{0, \pm} - \phi_{0,\pm}\|\leq \delta_1\}
\P\{\|\widetilde T _n (\phi_{\pm}) - (\phi-\phi(a_n))\|_{ I_n(\eta)}\leq \delta - 2 \delta_1 - \delta_2, n \in J \biggr\} \notag \\
&\geq \P\biggl\{A_\epsilon, \|T_{0, \pm} - \phi_{0,\pm}\|\leq \delta_3 \}
\P\{\|\widetilde T _n (\phi_{\pm}) - (\phi-\phi(a_n))\|_{ I_n(\eta)}\leq \delta_3, n \in J \} \nonumber .
\end{align}

We choose $\delta_3 < \delta - 2 \delta_1 - \delta_2$ and $\delta_3 < \delta_1$, independent of $\eta,$ so that we can set $\delta_3 \to 0 $ while retaining $\eta.$
The upper bound is derived to be
\begin{align}
&\P\biggl\{\|T_{0, \pm} - \phi_{0,\pm}\|\leq \delta, \|T- \phi\|\leq \delta \biggr\} \notag \\
\label{3.3.2}
&\leq \P\biggl\{\|T_{0, \pm} - \phi_{0,\pm}\|\leq 2\delta, \|T- \phi - (T- \phi)(a_n)\|_{\cup I_n(\zeta)}\leq 2\delta\biggr\}. 
\end{align}
Similarly,
\begin{align}
&\P\biggl\{\|T_{0, \pm} - \phi_{0,\pm}\|\leq 2\delta, \|\widetilde T _n(\phi_{\pm})- (\phi - \phi(a_n))\|_{\cup I_n(\zeta)}\leq 4\delta \biggr\} \notag \\
&= \P\biggl\{\|T_{0, \pm} - \phi_{0,\pm}\|\leq 2\delta\} \P\{\|\widetilde T _n(\phi_{\pm})- (\phi - \phi(a_n))\|_{\cup I_n(\zeta)}\leq 4\delta \biggr\}.\nonumber 
\end{align}

Note that for the upper and lower bounds, (\ref{3.3.1}) and (\ref{3.3.2}) are similar in the sense that $\eta$ and $\delta_1$ are independent and $\zeta$ and $2\delta$ are independent. Thus, in the following subsections, we can first set $\delta_1,2\delta \to 0$ and then set $\eta, \zeta \to 0$.

\subsubsection{Step 2: Exponentially equivalence and contraction principle}

By Lemmas \ref{lem3.1} and \ref{lem3.2}, we can replace $T_{\pm}$ due to $\widetilde T_{\pm}$ by the exponential equivalence.
Thus, we can replace $T$ by some simpler $T_n$ on each $I_n \subset P^c.$ 
Recall the definition of $P$, which is in $(\ref{3.3.3})$
 on $I_n,$ we have
$\phi_- - \eta > \phi$ or $\phi > \phi_+ + \eta.$
Therefore, guaranteed by the previous lemmas,
we construct the exponentially equivalent $\widetilde T_{n,0},$ which is
\begin{equation}
\widetilde T_{n,0} (t) - \widetilde T_n(a_n) =  \frac 1 {\sqrt \gamma}W (a_n) + \sup_{[a_n,t]} (T_- - \frac 1 {\sqrt \gamma}W )_+; \nonumber 
\end{equation}
if $\phi > \phi_+ + \eta,$ on $[a_n, b_n].$
\begin{equation}
\widetilde T_{n,0} (t) - \widetilde T_n(a_n) =  \frac 1 {\sqrt \gamma}W (a_n) - \sup_{[a_n,t]} (- T_+ + \frac 1 {\sqrt \gamma}W )_+; \nonumber 
\end{equation}
if $\phi < \phi_- - \eta,$ on $[a_n, b_n]$.

Suppose that we replace $T_-$ and $T_+$ by $\phi_-$ and $\phi_+$. We define $\widetilde T_n$ as
\begin{equation}
\widetilde T_n (t) - \widetilde T_n(a_n) =  \frac 1 {\sqrt \gamma}W (a_n) + \sup_{[a_n,t]} (T_- - \frac 1 {\sqrt \gamma}W )_+; \nonumber 
\end{equation}
if $\phi > \phi_+ + \eta,$ on $[a_n, b_n],$ and 
\begin{equation}
\widetilde T_n (t) - \widetilde T_n(a_n) =  \frac 1 {\sqrt \gamma}W (a_n) - \sup_{[a_n,t]} (- T_+ + \frac 1 {\sqrt \gamma}W )_+; \nonumber 
\end{equation}
if $\phi < \phi_- - \eta,$ on $[a_n, b_n]$.
Next, we will show that  
\begin{equation}
\label{norm_comparison}
\| \widetilde T_{n,0} - \widetilde T_n\|_{[a_n,b_n],\infty}
\leq \| \sup_{[a_n,t]} (- T_\pm + \frac 1 {\sqrt \gamma}W )_+ -  \sup_{[a_n,t]} (- \phi_\pm + \frac 1 {\sqrt \gamma}W )_+ \|_{\infty}
\leq \sup_{[a_n,b_n]} \| T_\pm - \phi_\pm \|_{\infty} .\nonumber 
\end{equation}
There exist $t_1, t_2 \leq t$ such that
\begin{align*}
&\sup_{[a_n,t]} (- T_\pm + \frac 1 {\sqrt \gamma}W )_+ =  (- T_\pm + \frac 1 {\sqrt \gamma}W )_+(t_1), \\
&\sup_{[a_n,t]} (- \phi_\pm + \frac 1 {\sqrt \gamma}W )_+ =  (- \phi_\pm + \frac 1 {\sqrt \gamma}W )_+(t_2) .
\end{align*}
By symmetry, we assume $t_1 \geq t_2$. If the first term is larger, i.e.,
$$
\sup_{[a_n,t]} (- T_\pm + \frac 1 {\sqrt \gamma}W )_+ \geq \sup_{[a_n,t]} (- \phi_\pm + \frac 1 {\sqrt \gamma}W )_+,
$$
then
\begin{align*}
\biggl\| \sup_{[a_n,t]} (- T_\pm + \frac 1 {\sqrt \gamma}W )_+ -  \sup_{[a_n,t]} (- \phi_\pm + \frac 1 {\sqrt \gamma}W )_+ \biggr\|_{\infty}
&= T_\pm(t_1) - \phi_\pm(t_2)
\leq T_\pm(t_1) - \phi_\pm(t_1) \\
&\leq \sup_{[a_n,t]} \biggl\| T_\pm - \phi_\pm \biggr\|_{\infty}.
\end{align*}
If the second term is larger, i.e.,
$$
\sup_{[a_n,t]} (- T_\pm + \frac 1 {\sqrt \gamma}W )_+ \leq \sup_{[a_n,t]} (- \phi_\pm + \frac 1 {\sqrt \gamma}W )_+,
$$
then 
\begin{align*}
\biggl\| \sup_{[a_n,t]} (- T_\pm + \frac 1 {\sqrt \gamma}W )_+ -  \sup_{[a_n,t]} (- \phi_\pm + \frac 1 {\sqrt \gamma}W )_+ \biggr\|_{\infty}
&= \phi_\pm(t_2) - T_\pm(t_1) \\
&\leq \phi_\pm(t_2) - T_\pm(t_2) \\
&\leq \sup_{[a_n,t]} \biggl\| T_\pm - \phi_\pm \biggr\|_{\infty}.
\end{align*}
Thus we proved (\ref{norm_comparison}).

Based on the above, for $\delta_1 < \delta_2 < \delta,$ we have:
\begin{align*}
& \P\{\|T_{0, \pm} - \phi_{0,\pm}\|\leq \delta, \|\widetilde T- \phi\|_{\cup I_n}\leq \delta \}  \notag \\
\geq &
\P\{\|T_0 - \phi_0\|\leq \delta_1, \|\widetilde T_{\pm} - \phi_{\pm}\|\leq \delta_2,\|\widetilde T_{n,0}- \phi\|\leq \delta_3, n \in J\} \notag \\
\geq &
\P\{\|T_0 - \phi_0\|\leq \delta_1, \|\widetilde T_{\pm} - \phi_{\pm}\|\leq \delta_2 - \delta_1,\|\widetilde T_n- \phi\|\leq \delta - \delta_2, n \in J\} \notag . \\
\end{align*}
Therefore, we can apply the contraction principle for the mapping:
$$\{W_0, W_{\pm}, W\} \to \{T_0, \widetilde T_{\pm}, \widetilde T_n , n\in J\}$$
and obtain
\begin{align}
&\lim_{\delta' \to 0} \lim_{\gamma \to \infty}
\P\biggl\{\|T_{0, \pm} - \phi_{0,\pm}\|\leq \delta', \|T- \phi\|_{\cup I_n}\leq \delta' \biggr\}  \notag \\
= &\lim_{\delta' \to 0} \lim_{\gamma \to \infty}
\P\biggl\{\|T_0 - \phi_0\|\leq \delta', \|\widetilde T_{\pm} - \phi_{\pm}\|\leq \delta',\|T_n- \phi\|\leq \delta', n \in J\biggr\} \notag \\
= &- \frac 1 2 \biggl(\int_0^1 \dot{\phi_0}^2
+ \int_{\phi_+ > \phi_0} \dot{\phi_+}^2 + \int_{\phi_+ = \phi_0} (\dot{\phi_+}_+)^2
+ \int_{\phi_- > \phi_0} \dot{\phi_-}^2 + \int_{\phi_- = \phi_0} (\dot{\phi_-}_-)^2  \notag \\
&+ \int_{\{\phi_- < \phi < \phi_+\} \cap \cup I_n(\eta)}  \dot{\phi}^2
+ \int_{\{\phi_- = \phi < \phi_+\}  \cap \cup I_n(\eta)} \dot{\phi_-}^2
+ \int_{\{\phi_- < \phi = \phi_+\}  \cap \cup I_n(\eta)} \dot{\phi_+}^2\biggr)  \notag \\
= &I(\phi) + I(\phi_+) + I(\phi_-)+ I(\phi,\eta).
    \end{align}
Since the above equation holds for all $\eta >0,$ $\cup I_n(\eta)$
would be the complement of $\{\phi_- = \phi =\phi_+\}.$
If we set $\eta \to 0,$
we have
\begin{equation}
\lim_{\eta \to 0}I(\phi,\eta) = I(\phi) =
\int_{\{\phi_- < \phi < \phi_+\} }  \dot{\phi}^2
+ \int_{\{\phi_- = \phi < \phi_+\}} \dot{\phi_-}^2
+ \int_{\{\phi_- < \phi = \phi_+\}} \dot{\phi_+}^2.
\end{equation}
Therefore, 
\begin{align}
\liminf_{\delta \to 0} \lim_{\gamma \to \infty} \P
\geq I(\phi) + I(\phi_+) + I(\phi_-)+ I(\phi).
\end{align}
Thus, the lower bound is proved.

The upper bound is similar to obtain and it is the same as the lower bound. Therefore, it is the rate function. This is a good rate function, because by the Schilder's Theorem and contraction principle. We have proved our result for the four-particle case.

\subsubsection{General Case}
When $n = 3$, the proof in the four-particle case still applies. In general, we can prove the results recursively.

\section{Conclusion and Future Remarks} \label{sec_con}
In this work, we presented a large deviation principle for the Whittaker 2d growth model. 
An interesting future problem is to emulate the analysis to address the Whittaker 2d growth model with more general functions instead of the current exponential function. 
Different from Freidlin-Gartner formula, our rate function depends on the position of sample paths. This is due to the interactions of interlaced particles. 
In Appendix~\ref{sec_example}, we elaborate on two special cases, namely the Crossing and Strict Interlacing of sample paths, where the proof can be simplified.
In Appendix~\ref{sec_existence}, we show the existence and uniqueness of the strong solution for technical completeness.


\section*{Acknowledgment}
The authors sincerely thank Amir Dembo for his advice on the topic. The authors are also grateful to Tianyi Zheng, Mykhaylo Shkolnikov, Vadim Gorin, and Ivan Corwin for helpful discussions.

\appendix

\section{Crossing or Strict Interlacing Case}
\label{sec_example}
We calculate the rate functions for two special cases.
The crossing case is where the positions of the sample paths are reversed and cause the drift term of Whittaker 2d growth model to explode with the scaling factor $\gamma.$
The strict interlacing case is where the positions of the sample paths are interlaced and cause the drift term to distinguish exponentially with $\gamma.$

\subsection{Crossing Case}
\begin{thm}[Infinite Convergence Rate]
Assume that there exist $n,k$ and $t\in [0,1]$ such that
$\phi_{n,k}(t) > \phi_{n,k+1}(t).$
Then the scaled system at the beginning of Section~\ref{sec_main} 
has an infinite rate function as $\gamma \to \infty,$ or equivalently, 
$$\frac{1}{\gamma} \lim_{\delta \to 0} \lim_{\gamma \to \infty} \log  \mathbb{P}\bigl\{\|T_{n,k} -\phi_{n,k}\|\leq \delta, 1 \leq k \leq n \leq N\bigr\} = -\infty. $$

\end{thm}

\begin{proof}
If we have $\phi_{n,k}(t) > \phi_{n,k-}(t),$ or $\phi_{n,k-}$ does not exist,
we can prove the infinite convergence rate.
We focus on the case where
for all $t$ it holds that $\phi_{n,k}(t) > \phi_{n,k+}(t),$
$\phi_{n,k-}(t) \geq \phi_{n,k}(t) > \phi_{n,k+}(t).$
However, if this is the case, consider $\phi_{n-2, k-1}$,
we have $\phi_{n-1,k-1}(t)< \phi_{n-2,k-1}(t),$ or
$\phi_{n-2,k-1}(t)< \phi_{n-1,k}(t). $
So we can again apply the above arguments until $T_{n,k}$ hit the boundary,  meaning that one of $T_{n,k\pm}$ does not exist. This will reduce to the following simple case.

We consider $\phi_{n,k}(t) - \phi_{n,k+}\geq \eta,$
and $\phi_{n,k} - \phi_{n,k}\geq \eta, $ on the interval $I \subset [0,1].$
Recall that $\|T_{n,k} - \phi_{n,k}\| \leq \delta, $ for all $n,k$.
If we choose $\delta $ small enough, such that
$T_{n,k} - T_{n,k+} \geq \phi_{n,k} - \phi_{n,k+} - 2 \delta \geq \eta - 2 \delta >0.$
and $T_{n,k-} -  T_{n,k}\leq -\eta + 2\delta <0, $
then
\begin{align*}
&\P\biggl\{\|T_{n,k} -\phi_{n,k}\|\leq \delta, 1 \leq k \leq n \leq N\biggr\}  
\leq \P\biggl\{T_{n,k} -\phi_{n,k}\leq \delta, \quad \textrm{on } [a,b]\biggr\} \nonumber \\
&\leq \P\biggl\{T_{n,k}(a) + \frac 1 {\sqrt \gamma} (W_{n,k}(t) - W_{n,k}(a))+ (t-a)(e^{\gamma(\eta -2 \delta)}-1) f\leq \phi(t)+\delta,  \biggr\},
\end{align*}
which has an infinite convergence rate.
\end{proof}

\subsection{Strict Interlacing Case }
For the strict interlacing case, the main theorem \ref{mainthm} still applies. However, we give a succinct proof here just for intuition.
For $\{\phi_{n,k}\}$ which are strictly interlaced, i.e.,
$$\phi_{n-1,k-1}\geq \phi_{n,k}\geq \phi_{n-1,k},$$on $[0,1]$ for al $1\leq k \leq n \leq N$,
if $\phi_{n-1,k-1}$ or $\phi_{n-1,k}$ does not exist, the inequality disappears automatically.
We obtain the rate function by proving the following. 

\begin{thm}[Non-Intersecting Case]
\begin{eqnarray*}
\lim_{\delta \to 0}\lim_{\gamma \to \infty}
\frac{1}{\gamma} \log
\P\biggl\{\|T_{n,k}- \phi_{n,k}\|\leq \delta; 1\leq k \leq n \leq N \biggr\}
= \sum_{1 \leq k \leq n \leq N} I(\phi_{n,k}).
\end{eqnarray*}
where,
\begin{eqnarray*}
I(\phi_{n,k})= \left\{ \begin{array}{ll}
 \frac 1 2 \int_0^1 {\dot \phi_{n,k}}^2
& \quad \textrm{if $\phi_{n,k} \in \mathcal{AC}$ and $\phi_{n,k}(0)=T_{n,k}(0);$}\\
 +\infty & \quad\textrm{otherwise.}
\end{array}
\right.
\end{eqnarray*}
\end{thm}

\begin{proof}
Since all the inequalities are strict on $[0,1],$ there is a $\eta$ such that
$$\phi_{n-1,k-1}- \phi_{n,k} \geq \eta ,$$
$$\phi_{n,k} - \phi_{n-1,k}\geq \eta,$$ on $[0,1]. $
then we choose $\delta < \frac \eta 3,$
\begin{align*}
\bigl|e^{\gamma(T_{n,k-} - T_{n,k})} - e^{\gamma(T_{n,k}-T_{n,k+})}\bigr| 
& \leq e^{\gamma(T_{n,k-} - T_{n,k})} + e^{\gamma(T_{n,k}-T_{n,k+})} \nonumber \\
& \leq e^{\gamma(\phi_{n,k-} - \phi_{n,k}+ 2\delta) } +  e^{\gamma(\phi_{n,k}-\phi_{n,k+} + 2\delta)} \nonumber \\
& \leq e^{\gamma(-\eta + 2\delta) } +  e^{\gamma(-\eta + 2\delta)}  \to 0, 
\end{align*}
as $\gamma \to \infty.$
Therefore, since  $\eta$ is fixed, for $\delta$ small enough, we can find $\gamma $ large enough that 
$$ |e^{\gamma(T_{n,k-} - T_{n,k})} - e^{\gamma(T_{n,k}-T_{n,k+})}| \leq \frac \delta 2.
$$
Based on the above, we have following two inequalities:
\begin{align*}
&\P\biggl\{  \|T_{n,k}-\phi_{n,k}\|\leq \delta  \biggr\} 
\geq \P\biggl\{\|\frac 1 {\sqrt \gamma}W_{n,k}(t) -T_{n,k}(0) - \phi_{n,k} \|\leq \frac \delta 2\biggr\}. \\
&\P\biggl\{  \|T_{n,k}-\phi_{n,k}\|\leq \delta \biggr \} 
\leq \P \biggl\{\|\frac 1 {\sqrt \gamma}W_{n,k}(t) -T_{n,k}(0) - \phi_{n,k} \|\leq \frac 3 2 \delta \biggr\}.
\end{align*}
Denote the right-hand sides of the above two inequalities as $R_1(\gamma,\delta) $ and $R_2(\gamma,\delta)$, respectively. 
By the Freidlin-Weizell Theorem,
$$\lim_{\delta \to 0}\lim_{\gamma \to \infty} -\frac{1}{\gamma} \log R_1(\gamma, \delta) = \lim_{\delta \to 0}\lim_{\gamma \to \infty} -\frac{1}{\gamma} \log R_2(\gamma, \delta) = \sum_{1 \leq k \leq n \leq N} I(\phi_{n,k}).$$
The above gives an upper bound for the rate function. A similar arguments can be made to show that the lower bound is the same, which concludes the proof.
\end{proof}

\section{Existence and Uniqueness of the Strong Solution} \label{sec_existence}
\label{sec_existence}
To prove the existence of the strong solution, we consider the general formula 
\begin{equation}
\label{2.1}
dT_{k,j}= dW_{k,j} +(e^{T_{k-1,j}-T_{k,j}}-e^{T_{k,j}-T_{k-1,j-1}})dt,
\end{equation}
for $ 1 \leq j \leq k $,
where the terms vanish when certain indices become zero or negative.

\begin{remark}
The existence and uniqueness of the strong solution for $T_{k,1} $ and $T_{k,k}$ with $1 \leq k \leq N$ can be directly computed as
\begin{equation}
dT_{k,1} = dW_{k,1} + e^{T_{k-1,1} - T_{k,1}}dt,
\end{equation}
which is equivalent to the ordinary differential equation
\begin{equation}
\frac {d} {dt} S_{k,1} =  e^{S_{k-1,1} - S_{k,1}},
\end{equation}
where $S_{k-1,1} = T_{k-1,1} - W_{k,1}$, $ S_{k,1} = T_{k,1} - W_{k,1}.$

By solving the ODE, we have the unique strong solution:
\begin{equation}
T_{k,1}(t) = T_{k,1}(0) + W_{k,1}(t) + \log\biggl\{ 1+ \int_{0}^t e^{T_{k-1,1}(s)- W_{k,1}(s)- T_{k,1}(0)}ds  \biggr\}.
\end{equation}

Since $T_{1,1}$ is a Brownian motion, we can obtain the solutions of $T_{k,1}$  for all $k$ recursively.
The computation for $T_{k,k}$ is similar.
The proof given below with minor modifications is a rigorous proof for $T_{k,1}$ and $T_{k,k}$.

\end{remark}

\begin{thm}[Existence and Uniqueness of the Strong Solution]
\label{existence}
For the stochastic differential system, the strong solution exists and is unique.
\end{thm}


\begin{lem}
\label{lemma 2.1}
Let $W$ denote a standard Brownian motion, $Y_1, Y_2$ stochastic processes almost surely bounded by $C$ for all $t.$ We define $X$ as the strong solution of the following equation:
\begin{equation}
dX=dW+ (e^{Y_1-X}-e^{X-Y_2})dt,
\end{equation}
with initial condition $X(0) =C_0$.
Then,
\begin{equation}
\P\biggl\{\sup_{[0,t]}X^2 \geq L^2 \biggr\} \leq \frac {G(C_0, C)} {L^2 - C_1(C) T - C_0^2}
\end{equation}
for some $G(C_0, C)$ that is a positive increasing function in $C_0$ and $C$, and some positive function $C_1$.
\end{lem}

\begin{proof}[Proof of Theorem \ref{existence}]
Consider the auxiliary system

$$ dT_{1,1}=dW_{1,1}+ a_1dt,$$

for $k=2, 3. \dots , N$:
\begin{align*}
dT_{k,1}&= dW_{k,1}+(a_k+\exp\{\phi_{L_{k-1}}(T_{k-1, 1})- \phi_{L_{k}}(T_{k,1})\})dt \\
dT_{k,2}&= dW_{k,2}+(a_k+\exp\{\phi_{L_{k-1}}(T_{k-1,2})-\phi_{L_{k}}(T_{k,2})\}-\exp\{\phi_{L_{k}}(T_{k.2})-\phi_{L_{k-1}}(T_{k-1,1})\}) dt,  \\
\dots \notag \\
dT_{k,k-1}&= dW_{k,k-1}+(a_k+\exp\{\phi_{L_{k-1}}(T_{k-1,k-1})-\phi_{L_{k}}(T_{k,k-1})\}- \nonumber \\
&\quad\exp\{\phi_{L_{k}}(T_{k.k-1})-\phi_{L_{k-1}}(T_{k-1,k-2})\})dt, \\
dT_{k,k}&=dW_{k,k}+(a_k-\exp\{\phi_{L_{k}}(T_{k,k})-\phi_{L_{k-1}}(T_{k-1,k-1})\})dt,
\end{align*}
where,
$$\phi_L(x) = x 1_{[-L,L](x)} + (-L)1_{x < -L} + L 1_{x > L}.$$

For a fixed $L = (L_k, 1 \leq k \leq N) $,
the coefficients are uniformly Lipschitz and bounded, and thus it is known  \cite{DZ,G} that
the system admits a unique strong solution
$T_{k,j,L}$
with respect to the filtration $\mathcal{F} $ generated by $B_{k,j}, 1 \leq j \leq k \leq N.$

We will construct the solution to the original system by setting:
$T_{k,j}(t) = T_{k,j,L}(t), $ on $|T_{k,j,L}(s)|\leq L, $ for $0 \leq s \leq t$.
To show that this construction is possible, we need to show that
$$T_{k,j,L} = T_{k,j,L'} $$
for all $L'\geq L$.
We also need to show that the solutions will not go to infinity. 

It can be shown that the following inequalities hold.
\begin{align*}
&\P\biggl\{\sum_{k,j} \sup_{[0,T]} T_{k,i}^2 \geq \sum_{1 \leq i \leq k \leq N} L_k^2 \biggr\}  \\
&\leq \P\biggl\{ \sup_{[0,T]} T_{1,1}^2\geq L_1^2  \biggr\} + \sum_{1 \leq j\leq 2}\P\biggl\{ \sup_{[0,T]} T_{2,j}^2 \geq L_2^2, T_{1,1}^2 \leq L_1^2 \biggr\} + \dots \\
&\quad+ \sum_{1 \leq j \leq N}\P\biggl\{ \sup_{[0,T]} T_{N,j}^2\geq L_N^2, \sup_{[0,T]} T_{k,i}^2 \leq L_k^2, 1 \leq i \leq k \leq N-1 \biggr\}.
\end{align*}

Since the starting positions of $T_{k,j}$ has a uniform bound, namely
\begin{equation}
\sup_{k,j} |T_{k,j}(0)| \leq C,
\end{equation}
Lemma \ref{lemma 2.1} implies that for $1 \leq k \leq N,$
\begin{equation}
\P\biggl\{ T_{l,i}^2\geq L_k^2,T_{k,i}^2 \leq L^2, 1 \leq i \leq k \leq l-1 \biggr\}
\leq \frac {G(C,L_{k-1})} {{L_{k}^2 - C_1(L_{k-1})T - C^2}}.
\end{equation}
If we keep $L_N \gg L_{N-1} \dots \gg L_1$ as $L_k \to \infty,$
the right-hand side converges to 0.
Therefore, 
\begin{equation}
\lim_{L \to \infty} \P\biggl\{ \sum_{1 \leq i \leq k \leq N} T_{k,i}^2 \geq L^2 \biggr\}=0,
\end{equation}
which further implies 
\begin{equation}
\label{2.2}
\sum_{1 \leq i \leq k \leq N} \sup_{[0,T]} T_{k,i}^2
< \infty,
\end{equation}
almost surely. Hence, we proved the existence of the strong solution.

Next, we prove the uniqueness of the solution. 
Set $T_L = \inf\{t: \sum_{k,j}T_{k,j}^2 \leq L^2\}.$
Note that (\ref{2.2}) is equivalent to
\begin{equation}
\lim_{L \to \infty}\P\{T_L \leq T \} = 0.
\end{equation}
On $\{T_{L} \geq T\},$
we have $ \sum_{k,j}T_{k,j}^2 \leq L^2, $ for all $t \in [0,T].$ 
Thus, the SDE system in the beginning of this proof with cutoff $\phi_L$ is equal to the equation without cutoff.
Consequently, if we choose $L$ and $L'$ large enough, the solutions satisfy the same equation on $\{T_{L} \geq T\}$ and $\{T_{L'} \geq T\}$, or equivalently $T_{k,j,L} = T_{k,j,L'} $ for $1 \leq j \leq k \leq N $. 
Thus we conclude that the strong solution is unique.

\end{proof}

\begin{proof}[Proof of Lemma~\ref{lemma 2.1}]
Consider the auxiliary equation 
\begin{equation}
dX=dW+ \bigl(\exp\{Y_1-\phi_L(X)\}-\exp\{\phi_L(X)-Y_2\}\bigr)dt,
\end{equation}
where $\phi_L (x) = x 1_{-L\geq x \leq L } + (-L) 1_{x < -L} + L 1_{x >L}. $
Then the coefficients are Lipschitz, and it is sufficient to study the probability of $\P ( X^2 \geq L^2).$

By Ito's formula,
\begin{equation}
\label{2.3}
dX^2= 2XdW+ (1+2X(e^{Y_1-X}-e^{X-Y_2}))dt  \leq 2XdW+ C_1 dt,
\end{equation}
where we choose 
\begin{equation}
\label{2.4}
C_1 = 1+ \sup_{x \in \mathbf{R},| Y_1 | <C, |Y_2| <C }2X(e^{Y_1-X} - e^{X-Y_2})
\leq 1+ \sup_{t > 0 } 2te^{C-t}.
\end{equation}

By the estimate (\ref{2.3}) we have the estimates:
\begin{align}
\label{2.5}
\P\biggl\{\sup_{ 0 \leq t \leq T}X^2(t) \geq L^2\biggr\}
&\leq \P\biggl\{C_0^2 + \sup_{ 0 \leq t \leq T}\int_0^t (2XdW + C_1 dt) \geq L^2\biggr\} \\ \notag
&\leq \P\biggl\{ \sup_{ 0 \leq t \leq T }\int_0^t 2XdW \geq L^2 - C_1T - C_0^2\biggr\} \\ \notag
&= 2\cdot \P\biggl\{ \int_0^T 2XdW \geq L^2 - C_1T - C_0^2\biggr\} \\ \notag
&\leq 2 \cdot 4  (L^2-C_1T-C_0^2)^{-1} \sup_{0 \leq t \leq T} \E\biggl\{\int_0^T X^2 ds\biggr\}.
\end{align}

To obtain an estimate for $E\{X^2\}$, we have from (\ref{2.3}) 
\begin{align*}
E\{X^2(t)\}
&\leq C_0^2 + C_1 t +  E\biggl\{\int_0^t 2X dW\biggr\} = C_0^2 + C_1 t .
\end{align*}
Note that $\int_0^t 2X dW$ is a martingale. By Gronwall's Inequality,
\begin{align}
\label{2.6}
E\biggl\{\int_0^T X^2\biggr\}
&\leq  C_0^2 T+ \frac{1}{2}C_1^2 T^2 = G(C_0, C_1).
\end{align}
Combining (\ref{2.5}) and (\ref{2.6}), we obtain the result
\begin{equation*}
\P\biggl\{\sup_{[0,t]}X^2 \geq L^2 \biggr\} \leq \frac {G(C_0, C)} {L^2 - C_1(C) T - C_0^2},
\end{equation*}
where $C_1$ is given by (\ref{2.4})
and is increasing in both $C_0$ and $C$. This proves the inequality in this lemma.
\end{proof}



\bibliographystyle{imsart-number} 
\bibliography{LDP.bib}

\end{document}